\documentclass[amscd,amssymb,verbatim,12pt]{amsart}
\newif\ifAMS
\IfFileExists{amssymb.sty}
  {\AMStrue\usepackage{amssymb}}
  {\usepackage{latexsym}}
  \usepackage{enumerate}
\theoremstyle{plain}

\newtheorem*{profile}{Theorem \ref{T:profile}}

\newtheorem{Thm}{Theorem}[section]
\newtheorem{Cor}[Thm]{Corollary}

\theoremstyle{definition}
\newtheorem{Def}{Definition}

\theoremstyle{remark}

\newcommand{\interior}{^{ \kern-5pt ^\circ}}
\newcommand {\bd}{\partial}

\begin{document}
\title
{Growth and isoperimetric profile of planar graphs}

\author
{Itai Benjamini}
\author
{Panos Papasoglu }


\email {itai.benjamini@weizmann.ac.il} \email {papazoglou@maths.ox.ac.uk}

\address
{Dept. of Mathematics, Weizmann Institute, Rehovot, 76100, Israel}

\address
{Mathematical Institute, University of Oxford, 24-29 St Giles',
Oxford, OX1 3LB, U.K.  }

\begin{abstract} Let $\Gamma $ be a planar graph such that the
volume function of $\Gamma $ satisfies $V(2n)\leq CV(n)$ for some
constant $C>0$. Then for every vertex $v$ of $\Gamma $ and $n\in
\Bbb N$, there is a domain $\Omega $ such that $B(v,n)\subset
\Omega$, $\bd \Omega \subset B(v, 6n)$ and $|\bd \Omega | \precsim
n$.

\end{abstract}
\maketitle
\section{Introduction}

Let $\Gamma $ be a locally finite graph. If $a$ is a vertex of
$\Gamma $ we denote by $B(a,n)$ the ball of radius $n$ centered at
$a$. If $B$ is a graph, we denote by $|B|$ the number of vertices
of $B$.

Define the \emph{growth function} of $\Gamma $ at $a$ by
$$V(a,n)=|B(a,n)|$$
A graph $\Gamma $ is \emph{doubling} if there is a
constant $C>0$ such that for all $a,b\in Vert(\Gamma )$ and $n\in
\Bbb N$, $V(a,2n)\leq CV(b,n)$. We say then that $C$ is a
\emph{doubling constant} for $\Gamma $.

If $\Gamma $ is a doubling graph then the degree of
vertices is uniformly bounded.

Note that for any $d \geq 1$, there are planar graphs with the doubling property for which,
for any $a$ and $n$, $V(a,n)$ is of order $n^d$, see e.g. the last section of \cite{BS}.

Say that a graph $\Gamma $ corresponds to a \emph{tessellation}
of $\Bbb R ^2$ if there is a $k\in \Bbb N$ such that all
components of $\Bbb R^2-\Gamma $ are bounded regions with at most
$k$ sides.
\begin{Def}

Let $(X,d)$ be a metric space. An $\epsilon $-\emph{net} $N$
of $X$ is a set such that $d(v_1,v_2)> \epsilon $ for all
$v_1,v_2\in N$ and $N$ is maximal set with this property.
\end{Def}
We remark that if $N$ is an $\epsilon $-net of $X$ then $X$ is
contained in the $\epsilon $-neighborhood of $N$.

\begin{Def}

Let $\Omega $ be a subgraph of a graph $\Gamma $. Then we denote
by $\bd \Omega $ the set of all vertices of $\Gamma - \Omega $
which have a neighbor in $\Omega $.
\end{Def}

Our main result is the following:

\begin{profile} Let $\Gamma $ be a doubling planar graph. Then there is a constant $\alpha $ so that
for every vertex $v$ of $\Gamma $ there is a finite domain $\Omega
$ such that $B(v,n)\subset \Omega $, $\bd \Omega \subset B(v,6n)$
and $|\bd \Omega |\leq \alpha n$.
\end{profile}

Krikun \cite{Kr} has shown a similar theorem for the uniform
infinite planar triangulation (UIPT) introduced in \cite{AS}. The
volume doubling property does not hold for the UIPT, still an
asymptotic version should hold: for any vertex $v$, for large
enough $n$, $B(v,2n)$ contains order $1$ disjoint balls of radius
$n/2$ a.s. and the proof below will adapt to give Krikun's result. The
asymptotic volume growth of balls in the UIPT is order $n^4$, up to polylog's, see
\cite{A}, thus a weaker result with a polylog correction follows from our result.
For  the uniform infinite planar quadrangulation a sharp volume growth
estimate is known which implies the asymptotic volume doubling \cite{Kr2, CMM}.
For a detailed study of the geometry of the
uniform infinite planar quadrangulation and a matching lower bound
on the small cuts see \cite{CMM}.
\medskip

Our result shows that the volume and the isoperimetric profile
function are related for planar graphs. We define here the
isoperimetric profile function of a graph in a similar way as for
Riemannian manifolds:

\begin{Def} Let $\Gamma $ be a locally finite graph and let $V(n)$
be the volume function of $\Gamma $. Then the \emph{isoperimetric
profile} function of $\Gamma $, $I_{\Gamma }:\Bbb N \to \Bbb N $
is defined by:
$$ I_{\Gamma }(n)=\underset {\Omega} { \inf } \{|\bd \Omega|:
\Omega\subset \Gamma,\, |\Omega|\leq n \}$$ where $\Omega $ ranges
over all subgraphs of $\Gamma$.
\end{Def}
From theorem \ref{T:profile} we obtain the following:
\begin{Cor} Let $\Gamma $ be a doubling planar graph with volume
function $V(n)$ and isoperimetric profile function $I_{\Gamma
}(n)$.

Let $\varphi (n)=\inf \{k: \,V(k)\geq n \}$. Then there is a
constant $\alpha  $ such that $$I_{\Gamma }(n)\leq \alpha \varphi
(n) \text { for all } n\in \Bbb N$$
\end{Cor}

The same result holds for Riemannian metrics on the plane or on
the plane with holes (again assuming the doubling property holds
for the metric) and our proof extends to this setting as well (see
\cite{P} for more information on isoperimetric profiles of planes
and planes with holes).

\subsection{Idea of the proof}

Here is a sketch of the proof of theorem \ref{T:profile}. Let
$v$ be any vertex of $\Gamma $. Consider the balls
$B(v,n),B(v,3n)$. Let $N$ be an $n$-net of $\bd B(v,3n)$. For each
vertex $w$ of $N$ consider $B(w,n/2)$. Note that all such balls
are disjoint since $N$ is an $n$-net. Also all these balls are
contained in $B(v,4n)$. So, by the doubling property, we can have
only boundedly many such balls, that is $|N|\leq \beta $, where
$\beta $ does not depend on $n$. Consider now the balls $B(w,2n)$
for all $w\in N$. $\bd B(v,3n)$ is contained in the union of these
balls. Construct a closed curve that `blocks' $v$ from
infinity as follows: if $w_1,w_2\in N$ are such that
$d(w_1,w_2)\leq 2n$ then we join them by a geodesic. So  replace
$\bd B(v,3n)$ by the `polygonal line' that we define using
vertices in $N$. This `polygonal line' blocks $v$ from infinity
and has length at most $2n \beta $. In the next section we make
precise this idea. There are some technical issues to take care
of, for example $\bd B(v,3n)$ might not be connected (and could
even have `large gaps') and the geodesic segments have to be
chosen carefully. In particular the constants  obtained will be
slightly different from the ones in this this sketch.

\section{Growth and profile}

\begin{Def}
Let $\Gamma $ be an infinite planar locally finite graph. We say
that an embedding of $\Gamma $ in the plane is \emph{tame} if for
any bounded subset $A$ of the plane, $\Gamma -A$ has at least one
connected component of infinite diameter. It is easy to see that
any infinite planar locally finite graph admits a tame embedding
to $\Bbb R ^2$.

We see now $\Gamma $ as embedded in the plane by a tame embedding.
Let $B=B(v,r)$ be a ball of $\Gamma $ and let $U$ be the unbounded
connected component of $\Bbb R^2-B$. Define the \emph{contour}
of $B(v,r)$ to be the graph $B\cap \bar U$. Note that the
contour of $B(v,r)$ is connected.
\end{Def}

Recall some basic facts about winding numbers (see e.g.
\cite{Fu}, ch. 3 for a definition and basic properties of winding
numbers).

Let $\gamma $ be a closed curve on the plane and $v$ be a point
that does not lie on $\gamma $. If the winding number $W(\gamma,
v)$ is non zero then $v$ lies in a bounded component of $\Bbb R
^2-\gamma $. If $\gamma =\gamma _1\cup \gamma _2$ where $\gamma
_1,\gamma _2$ are paths with the same endpoints $a,b$ and if
$\delta $ is another path with endpoints $a,b$ then
$$W(\gamma ,v)=W(\gamma _1\cup \delta, v)+ W(\gamma _2\cup \delta,
v) $$ If $\gamma $ is a closed curve then $W(\gamma, v)\in \Bbb
Z$. Here it will be convenient to consider the winding number
modulo 2, so in what follows by winding number of $\gamma $ around
$v$ we mean $W(\gamma, v)\,mod\,2$.

\begin{Thm}\label{T:profile} Let $\Gamma $ be a doubling planar graph.
Then there is a constant $\alpha $ so that
for every vertex $v$ of $\Gamma $ there is a finite domain $\Omega
$ such that $B(v,n)\subset \Omega $, $\bd \Omega \subset B(v,6n)$
and $|\bd \Omega |\leq \alpha n$.
\end{Thm}
\begin{proof} Consider a tame embedding of $\Gamma $ in the
plane. Let $F$ be the contour of $B(v,4n)$ with respect to this
embedding. Clearly $F$ is a connected subset of the plane.
`Parametrise' $F$ by a map $f:S\to F$ where $S$ is a graph
homeomorphic to the circle, $f$ sends edges to edges and every
vertex $w$ such that $d(w,v)=4n$, has a unique pre-image
$f^{-1}(w)$. Note that $F$ is not necessarily a simple closed
curve, for example consider the case that $\Gamma $ is a tree.
Denote by $B'$ the set of all vertices $w$ in $\Gamma $ such that
$d(w,v)=4n$.

It is possible that $v\in F=im\,f$. As this creates some technical
problems we modify $f$ slightly to avoid this, so that the curve
$f(S)$ goes around $v$. For example we may do this as follows:
Cyclically order the edges adjacent to $v$, as
$e_1=[v,v_1],...,e_k=[v,v_k]$ and introduce new edges
$e_1',...,e_k'$ so that the edge $e_i'$ joins $v_i,v_{i+1}$ $(i\in
\Bbb Z_k)$. One may assume that the $e_i'$s do not intersect the
interior of any edge of $\Gamma $ and that the only edges
contained in the interior of $e_1'\cup ...\cup e_k'$ are
$e_1,...,e_k$.

If for two successive edges $g_1,g_2$ of $S$, we have
$f(g_1)=e_i,f(g_2)=e_{i+1}$ modify $f$ so that $f(g_1\cup
g_2)=e_i'$.

After this modification the winding number of $f(S)$ around $v$ is
defined and it is equal to $1$. If $S_1$ is a subarc of $S$
define the \emph{weight} of $S_1$ to be $w(S_1)=|f(S_1)\cap B'|$.

Note that if $a,b\in B'$ then there are two subarcs $S_1,S_2$
of $S$ such that $S=S_1\cup S_2$ and $f(S_1),f(S_2)$ are paths
with endpoints $a,b$. If we further assume that $d(a,b)< 4n $ and
$\gamma $ is a geodesic in $\Gamma $ joining $a,b$ then $v\notin
\gamma $, so one of the closed curves $f(S_1)\cup \gamma $ and
$f(S_2)\cup \gamma $ has winding number $1$ around $v$, while the
other has winding number $0$ around $v$.

If for any two vertices $a,b$ of $B'$, $d(a,b)>2n+1$ then the
balls $B(a,n)$, where $a\in B'$ are disjoint. So by the doubling
property $|B'|\leq C^3$, where $C$ is the doubling constant of
$\Gamma $. Therefore in this case one can take $\Omega =B(v,4n)$.

Otherwise pick $a,b$ in $B'$ with the following two properties:

1) $d(a,b)\leq 2n+1$.

2) There is a geodesic $\gamma $ in $\Gamma $ joining $a,b$ so
that for the subarc $S_1\subset S$ for which the closed curve
$f(S_1)\cup \gamma $ has winding number $1$ around $v$, $w(S_1)$ is
minimum possible.


If $f(S_1)\cap B'=\{a,b\}$ we may take $\Omega $ to be the
component of $\Gamma -\gamma $ containing $v$ and the theorem is
proved.

Otherwise, observe that $f(S_1)\cup \gamma $ separates $v$ from
infinity. We now explain how to replace $f(S_1)$ by a curve of length
linear in $n$.

Consider an $n$-net $N$ of the set $B'\cap f(S_1)$. Observe
that if $x,y\in N$ the balls $B(x,n/2),B(y,n/2)$ are disjoint and
are contained in $B(v,9n/2)$. It follows that $|N|<C^4$.

Define a graph $\Delta $ with set of vertices equal to
$N\cup \{a,b\}$. Join two vertices $x,y$ by an edge if
$d(x,y)\leq 2n+1$.

We consider first the case that $\Delta $ is connected. Then there
is a simple path in $\Delta $ joining $a,b$. This simple path
corresponds to a path $\delta $ in $\Gamma $. We obtain $\delta $
by replacing edges of the path in $\Delta $ by geodesic paths of
length $\leq 2n+1$ in $\Gamma $. It follows that $$length (\delta
)\leq (2n+1)C^4.$$ Consider now the closed path $$p=\gamma \cup
\delta
$$

We claim that the winding number of $p$ around $v$ is $1$. Indeed
$\delta=\delta _1\cup ...\cup \delta _k$ where $\delta
_i=[a_i,a_{i+1}]$ are geodesic paths of $\Gamma $ corresponding to
edges of $\Delta $ as explained above and $a=a_1,\,b=a_{k+1}$.
Clearly there are $a_1',...,a_{k+1}'\in S_1$ such that
$f(a_i')=a_i$ and  $$\bigcup _{i=1}^k (a_i',a_{i+1}')=S_1.$$
The curve $$f(a_i',a_{i+1}')\cup \delta _i$$ is closed
and has winding number 0 around $v$. This follows from the
definition of $\gamma $ and the fact that the weight of
$(a_i',a_{i+1}')$ is smaller than the weight of $S_1$. By the
additivity property of winding numbers we see that by replacing
successively $f(a_i',a_{i+1}')$ by $\delta _i$ in $f(S_1)\cup
\gamma $ the winding number around $v$ remains the same. So the
winding number of $p$ around $v$ is $1$ and the claim is proved.

Let $\Omega $ be the component of $\Gamma -p$ containing $v$.
$\Omega $ clearly has all the properties required by the theorem.


Let's return now to the proof of the general case, so we don't
assume anymore that $\Delta $ is connected.

Let $\Delta _1$ be the connected component of $\Delta $ containing
$a$. Let $V_1$ be the vertex set of $\Delta _1$.

Let's denote by $a',b'$ the endpoints of $S_1$, so
$f(a')=a,f(b')=b$.

Let $a_2'\in S_1$ be such that
$$d(f(a_2'),V _1)\leq n$$ and $f[a',a_2']$ contains all
vertices of $f(S_1)$ that are at distance $\leq n$ from $V _1$.
Join $f(a_2')$ to a vertex of $\Delta _1$ by a geodesic $\gamma
_1$ of length $\leq n$. If $d_1=|V_1|$ then, as before, we see
that there is a path $\delta _1$ joining $a$ to $f(a_2')$ such
that $$length (\delta _1)\leq d_1(2n+1)$$ and the winding number
of $\gamma \cup \delta _1\cup f[a_2',b']$ around $v$ is $1$.

Let $b_2'$ be the first vertex of $S_1$ following $a_2'$ such that
$d(f(b_2'),v)=4n$. Consider an $n$-net $N_1$ of the set $B'\cap
f(a_2',b')$. Clearly $N_1\cup V_1$ is an $n$-net of the set
$B'\cap f(S_1)$ so $|N_1\cup V_1|\leq C^4$. As before define a
graph with vertex set $\{f(b_2')\}\cup N_1$ and   consider the
connected component, say $\Delta _2$, of this graph which contains
$f(b_2')$. If $V_2$ is the vertex set of $\Delta _2$, let $a_3'\in
S_1$ be such that
$$d(f(a_3'),V _2)\leq n$$ and $f[b_2',a_3']$ contains all
vertices of $f(S_1)$ that are at distance $\leq n$ from $V _2$.
Join $f(a_3')$ to a vertex of $\Delta _1$ by a geodesic $\gamma
_2$ of length $\leq n$. If $d_2=|V_2|$ then, as before, we see
that there is a path $\delta _2$ joining $f(b_2')$ to $f(a_3')$
such that
$$length (\delta _2)\leq d_2(2n+1)$$ and the winding number of
$$\gamma \cup \delta _1\cup f[a_2',b_2']\cup \delta _2\cup
f[a_3',b']$$ around $v$ is 1.

We continue in the same way. Consider an $n$-net $N_2$ of
$f[a_3',b']\cap B'$, define similarly a graph $\Delta _3$ and a
path $\delta _3$.  Clearly this procedure terminates and
eventually we produce paths $\delta _1,...,\delta _k$ such that:

1) The winding number of $$\gamma \cup \delta _1\cup
f[a_2',b_2']\cup ...\cup f[a_k',b_k']\cup \delta _k$$ around $v$
is 1.

2)$f(a_i',b_i')\cap B'=\emptyset $ for all $i=2,...,k$.

3) $\sum _{i=1}^k length (\delta _i)\leq C^4(2n+1)$.

4) The paths $\gamma ,\delta _1,\delta _2,...,\delta _k$ are
contained in $B(v,6n)-B(v,n)$.

Let  $\Omega $ be the connected component of $\Gamma - (\gamma
\cup \delta _1\cup ...\cup \delta _k)$ containing $v$. By the
definition of $\Omega $, $B(v,n)\subset \Omega $. By property $3$
above $|\bd \Omega |\leq (C^4+1)(2n+1)$.

\end{proof}

\noindent
{\bf Remarks:}
\medskip

\noindent
1. If $\Gamma $ corresponds to a tessellation then, the
proof of the theorem shows that there is a simple closed curve of
length $\sim n$ in $B(v,6n)$ that separates $B(v,n)$ from
infinity. In the random context this has been shown by Krikun
(\cite {Kr}, \cite{Kr2}) for triangulations and quadrangulations.
Our proof applies regardless of the `shape' of the regions of the
tessellation and with uniform bounds.
\medskip

\noindent
2. Given any $\epsilon >0$ one may easily adapt the proof of the
theorem to produce regions $\Omega $ with the same properties such
that $$\bd \Omega \subset B(v,(1+\epsilon)n)$$ for $n$
sufficiently large. Of course in this case one obtains $|\bd
\Omega|\leq C(\epsilon)n$ and $C(\epsilon)\to \infty $ as
$\epsilon \to 0$.

\section{Further comments}

A graph $G$ admits unform volume growth $f(n)$, if there are $0 < c < C < \infty$,
so that for all $n$, any ball of radius $n$ in $G$ satisfies,
$$
c f(n) < |B(v,n)| < Cf(n).
$$

For planar graph admitting arbitrarily large uniform polynomial growth see e.g. \cite{BS}.
It is conjectured that planar graphs of uniform polynomial growth are recurrent for the simple random walk.
It is also conjectured \cite{AS} that the UIPT is recurrent, see also \cite{BS}.
By the Nash-Williams sufficient condition for recurrence, it is enough to find infinitely many disjoint cutsets
$\{C_i\}$ separating the root from infinity. so that $\sum |C_i|^{-1} = \infty$.
The theorem above is a step in that direction, still we don't know if planar graphs of uniform polynomial growth
admits such cutsets? Maybe not.
\medskip

Assume $G$ is a planar graph of uniform polynomial growth $n^d, d > 2$, by the theorem above $G$ admits bottlenecks.
This suggests the {\em conjecture} that simple random walk on $G$ will be subdiffusive,
as it will spend a lot of time in domains with small boundary before exciting.
That is, the expected distance to the root by time $t$
is bounded by $t^{\alpha}$ for $\alpha < 1/2$. Does $\alpha = d^{-1}$?
Subdiffusivity was recently established for the UIPT by Nicolas Curien and the first author.
\medskip

What about a high dimensional generalization?
A $d$-sphere packing is a collection of $d$-dimensional
balls with disjoint interiors. Associated to the packing an unoriented graph $G = (V,E)$ called the {\it d-tangency
graph}, where vertices corresponds to the $d$-balls and edges  are  between any two tangent balls, see \cite{BC}.
Is it the case that for any $d$ a $d$-tangency graph with the doubling property
admits cutsets outside a ball of radius $n$ of size $n^{d-1}$?
\medskip

Let $G$ be a planar triangulation which is doubling and further assume all balls has growth $r^d, d >2$
up to a multiplicative constant.
Is there such  $G$ for which  all complements of balls are connected, for all balls?
Or as in the UIPT, the complements of some balls  admit several connected components, some of size proportional to the
ball?

\end{document}
\bye